\documentclass[3p]{elsarticle}
\usepackage{lineno,hyperref}
\usepackage{graphicx,color}
\usepackage{amsmath,amssymb}
\usepackage{amsthm}

\journal{\ldots}

\newtheorem{theorem}{Theorem}[section]
\newtheorem{remark}{Remark}[section]
\newtheorem{corollary}{Corollary}[section]
\numberwithin{equation}{section}

\renewcommand{\Re}{\mathop{\mathrm{Re}}\nolimits}
\renewcommand{\Im}{\mathop{\mathrm{Im}}\nolimits}
\newcommand{\sgn}{\mathop{\mathrm{sgn}}\nolimits}
\renewcommand{\div}{\mathop{\mathrm{div}}\nolimits}

\definecolor{dgreen}{rgb}{0,0.5,0}

\definecolor{dblue}{rgb}{0,0,0.6}

\definecolor{dred}{rgb}{0.784,0,0}

\definecolor{delete}{cmyk}{0.5,0,0,0}

\begin{document}
\begin{frontmatter}
\title{Uncertainty Relations in the Framework of Equalities}
%
%
%
\author[WAP]{Tohru Ozawa}
\author[WPP]{Kazuya Yuasa}
\address[WAP]{Department of Applied Physics, Waseda University, Tokyo 169-8555, Japan}
\address[WPP]{Department of Physics, Waseda University, Tokyo 169-8555, Japan}

\begin{abstract}
We study the Schr\"odinger-Robertson uncertainty relations in an algebraic framework.
Moreover, we show that some specific commutation relations imply new equalities, which are regarded as equality versions of well-known inequalities such as Hardy's inequality.
\end{abstract}

\end{frontmatter}

\section{Introduction}
In this paper, we study the Schr\"odinger-Robertson uncertainty relations as corollaries of equalities in a scalar product space.
Moreover, we give a number of characterizations in the case where the associated inequalities are in fact equalities.
Our presentation is based exclusively on an algebraic observation on the standard Cauchy-Schwarz inequality and could presumably provide a clear and explicit understanding of uncertainty relations from the point of view of orthogonality.
As applications, we show that some specific commutation relations, in the Hilbert space $L^2(\mathbb{R}^n)$ of square integrable functions on the Euclidean space $\mathbb{R}^n$ of dimensions $n$, imply new norm equalities in $L^2(\mathbb{R}^n)$, which are regarded as equality versions of well-known inequalities such as dilation and Hardy type inequalities.
In particular, we give a method of recognizing Hardy type inequalities in the framework of commutation relations of operators.

This paper is organized as follows.
In Section \ref{sec:UncertaintyRelations}, we characterize the Schr\"odinger-Robertson uncertainty relations in the framework of equalities in a scalar product space.
In Section \ref{sec:Applications}, we give a number of examples of uncertainty relations on the basis of equalities in $L^2(\mathbb{R}^n)$.
In the Appendix, we summarize basic theorems on the Cauchy-Schwarz inequality in an algebraic setting.

Throughout the paper, $H$ denotes a complex vector space endowed with scalar product $(\,\cdot\,|\,\cdot\,):H\times H\ni(u,v)\mapsto(u|v)\in\mathbb{C}$, which is linear (resp.\ antilinear) in the first (resp.\ second) variable.
The associated norm is defined by $\|u\|=(u|u)^{1/2}$, $u\in H$.

There is a large literature on the uncertainty relations.
We refer the readers to \cite{ref:FollandSitaram1997,ref:FuruichiYanagi2012,ref:Dang2013,ref:KoYoo2014,ref:TawfikDiab2014} and references therein.

\section{Uncertainty Relations}
\label{sec:UncertaintyRelations}
Let $A$ and $B$ be symmetric operators in $H$ with domains $D(A)$ and $D(B)$, respectively.
In this and next sections, we use the terminology of operator theory (see \cite{ref:GustafsonSigal,ref:ReedSimonII} for instance).

According to Mourre \cite{ref:Mourre1981}, we define the commutator $[A,B]$ as a sesquilinear form on $H\times H$ by
\begin{equation}
([A,B]\varphi|\psi)
=(B\varphi|A\psi)
-(A\varphi|B\psi),\qquad
\varphi,\psi\in M\equiv D(A)\cap D(B).
\label{eqn:[A,B]}
\end{equation}
It coincides with the usual definition $AB-BA$ on $D(AB)\cap D(BA)$, which is smaller than $M$.
In this paper, we adopt the definition (\ref{eqn:[A,B]}) to avoid the domain problem as much as possible (see \cite{ref:Kosaki2006}) and assume that
\begin{equation}
M\neq\{0\}
\end{equation}
to avoid trivial cases.

Similarly, we define the anticommutator $\{A,B\}$ by
\begin{equation}
(\{A,B\}\varphi|\psi)
=(B\varphi|A\psi)
+(A\varphi|B\psi),\qquad
\varphi,\psi\in M.
\end{equation}
It is straightforward to verify that for all $\varphi\in M$,
\begin{align}
&\bullet\ \ %
([A,B]\varphi|\varphi)\in i\mathbb{R},
\\[2truemm]
&\hphantom{{}\bullet{}}\ \ %
([A,B]\varphi|\varphi)
=-2i\Im(A\varphi|B\varphi)
=2i\Re i(A\varphi|B\varphi)
\nonumber\\
&\hphantom{{}\bullet\ \ %
([A,B]\varphi|\varphi)}
=2i\Im(B\varphi|A\varphi)
=-2i\Re i(B\varphi|A\varphi).
\label{eqn:2.5}
\\[2truemm]
&\bullet\ \ %
(\{A,B\}\varphi|\varphi)\in\mathbb{R},
\\[2truemm]
&\hphantom{{}\bullet{}}\ \ %
(\{A,B\}\varphi|\varphi)
=2\Re(A\varphi|B\varphi)
=2\Im i(A\varphi|B\varphi)
\nonumber\\
&\hphantom{{}\bullet\ \ %
(\{A,B\}\varphi|\varphi)}
=2\Re(B\varphi|A\varphi)
=2\Im i(B\varphi|A\varphi).
\\[2truemm]
&\bullet\ \ %
(A\varphi|B\varphi)
=\frac{1}{2}
(\{A,B\}\varphi|\varphi)
-\frac{1}{2}
([A,B]\varphi|\varphi).
\\[2truemm]
&\bullet\ \ %
(B\varphi|A\varphi)
=\frac{1}{2}
(\{A,B\}\varphi|\varphi)
+\frac{1}{2}
([A,B]\varphi|\varphi).
\end{align}

We now summarize algebraic identities related to the Schr\"odinger-Robertson uncertainty relations.
The identities below are the direct consequences of the theorems in the Appendix.

\begin{theorem}
Let $\varphi\in M$satisfy $A\varphi\neq0$, $B\varphi\neq0$.
Then,
\begin{align}
\pm i([A,B]\varphi|\varphi)
&=\|A\varphi\|\|B\varphi\|\left(
2-\left\|
\frac{A\varphi}{\|A\varphi\|}
\mp i\frac{B\varphi}{\|B\varphi\|}
\right\|^2
\right),
\stepcounter{equation}\tag*{(\theequation)$_\pm$}
\label{eqn:2.10}
\\
\pm(\{A,B\}\varphi|\varphi)
&=\|A\varphi\|\|B\varphi\|\left(
2-\left\|
\frac{A\varphi}{\|A\varphi\|}
\mp\frac{B\varphi}{\|B\varphi\|}
\right\|^2
\right),
\stepcounter{equation}\tag*{(\theequation)$_\pm$}
\\
|(A\varphi|B\varphi)|
&=\frac{1}{2}\,\Bigl(
|([A,B]\varphi|\varphi)|^2
+|(\{A,B\}\varphi|\varphi)|^2
\Bigr)^{1/2}
\nonumber\\
&=\|A\varphi\|\|B\varphi\|
\left[
\left(
1-\frac{1}{2}\left\|
\frac{A\varphi}{\|A\varphi\|}
+e^{i\theta}
\frac{B\varphi}{\|B\varphi\|}
\right\|^2
\right)^2
+
\left(
1-\frac{1}{2}\left\|
\frac{A\varphi}{\|A\varphi\|}
\pm ie^{i\theta}
\frac{B\varphi}{\|B\varphi\|}
\right\|^2
\right)^2
\right]^{1/2}
\nonumber\\
&=\|A\varphi\|\|B\varphi\|
\left(
1-\frac{1}{2}\left\|
\frac{A\varphi}{\|A\varphi\|}
-[\sgn(A\varphi|B\varphi)]
\frac{B\varphi}{\|B\varphi\|}
\right\|^2
\right)
\label{eqn:2.12}
\end{align}
for any $\theta\in\mathbb{R}$, where $\sgn z=z/|z|$, $z\in\mathbb{C}\setminus\{0\}$, and $\sgn0=1$.
\end{theorem}

\begin{remark}
The standard uncertainty inequalities follow directly from equality (\ref{eqn:2.12}).
Indeed, equality (\ref{eqn:2.12}) implies the inequality
\begin{equation}
\|A\varphi\|\|B\varphi\|
\ge|(A\varphi|B\varphi)|
=\frac{1}{2}\,\Bigl(
|([A,B]\varphi|\varphi)|^2
+|(\{A,B\}\varphi|\varphi)|^2
\Bigr)^{1/2},
\label{eqn:SchroedingerIneq}
\end{equation}
which is called Schr\"odinger-Robertson uncertainty inequality \cite{ref:FuruichiYanagi2012,ref:HolevoText2012}.
It is further bounded from below by
\begin{equation}
\|A\varphi\|\|B\varphi\|
\ge\frac{1}{2}|([A,B]\varphi|\varphi)|,
\label{eqn:RobertsonIneq}
\end{equation}
which is known as the Robertson uncertainty inequality \cite{ref:QuntumInfoText-HayashiKimura}.\footnote{There are interesting developments on the Heisenberg uncertainty principle, beyond the Schr\"odinger-Robertson uncertainty relations. See for instance \cite{ref:Ozawa-PRA2003,ref:Ozawa-AnnPhys2004,ref:Branciard-PANS,ref:YuWatanabePhDThesis,ref:BranciardWhite-PRL2014,ref:OzawaEdamatsu-PRL2014,ref:BastosPRA2014,ref:RudolphPRA2014,ref:BuschLahtiWerner-RMP} and references therein.}
\end{remark}

Moreover, characterizations of extremizers are given by:
\begin{theorem}
Let $\varphi\in M$.
Then, the statements in each of the following Parts (1)--(5) are equivalent:
\begin{enumerate}
\item
\begin{enumerate}
\item $(\{A,B\}\varphi|\varphi)=\pm2\|A\varphi\|\|B\varphi\|$.
\medskip
\item $\|B\varphi\|A\varphi=\pm\|A\varphi\|B\varphi$.
\medskip
\item $(A\varphi|B\varphi)=\pm\|A\varphi\|\|B\varphi\|$.
\end{enumerate}

\medskip
\item
\begin{enumerate}
\item $i([A,B]\varphi|\varphi)=\pm2\|A\varphi\|\|B\varphi\|$.
\medskip
\item $\|B\varphi\|A\varphi=\pm i\|A\varphi\|B\varphi$.
\medskip
\item $(A\varphi|B\varphi)=\pm i\|A\varphi\|\|B\varphi\|$.
\end{enumerate}

\medskip
\item
\begin{enumerate}
\item $|(\{A,B\}\varphi|\varphi)|=2\|A\varphi\|\|B\varphi\|$.
\medskip
\item $([A,B]\varphi|\varphi)=0$, $|(A\varphi|B\varphi)|=\|A\varphi\|\|B\varphi\|$.
\medskip
\item $2\|B\varphi\|^2A\varphi=(\{A,B\}\varphi|\varphi)B\varphi$.
\medskip
\item $2\|A\varphi\|^2B\varphi=(\{A,B\}\varphi|\varphi)A\varphi$.
\end{enumerate}

\medskip
\item
\begin{enumerate}
\item $|([A,B]\varphi|\varphi)|=2\|A\varphi\|\|B\varphi\|$.
\medskip
\item $(\{A,B\}\varphi|\varphi)=0$, $|(A\varphi|B\varphi)|=\|A\varphi\|\|B\varphi\|$.
\medskip
\item $2\|B\varphi\|^2A\varphi=-([A,B]\varphi|\varphi)B\varphi$.
\medskip
\item $2\|A\varphi\|^2B\varphi=([A,B]\varphi|\varphi)A\varphi$.
\end{enumerate}

\medskip
\item
\begin{enumerate}
\item $|(A\varphi|B\varphi)|=\|A\varphi\|\|B\varphi\|$.
\medskip
\item $\|B\varphi\|A\varphi=[\sgn(A\varphi|B\varphi)]\|A\varphi\|B\varphi$.
\medskip
\item $\|B\varphi\|^2A\varphi=(A\varphi|B\varphi)B\varphi$.
\medskip
\item $\|A\varphi\|^2B\varphi=\overline{(A\varphi|B\varphi)}A\varphi$.
\end{enumerate}
\end{enumerate}
\end{theorem}

\section{Applications}
\label{sec:Applications}
In this section, we give a number of examples of commutation relations between operators in the Hilbert space $L^2(\mathbb{R}^n)$ of square integrable functions on $\mathbb{R}^n$ as well as related norm identities which are regarded as equality versions of well-known inequalities.
We follow the standard notation to denote a point in $\mathbb{R}^n$ by $x=(x_1,\ldots,x_n)\in\mathbb{R}^n$.
The associated Euclidean length is defined as $|x|=(x_1^2+\cdots+x_n^2)^{1/2}$.
The gradient operator is defined as $\nabla=(\partial_1,\ldots,\partial_n)$, where $\partial_j=\partial/\partial x_j$ is the partial differential operator in the $j$th direction.

\subsection{Momentum and Position Operators}
Let $A=-i\nabla$ and $B=x$.
More precisely, 
\begin{align}
A\varphi&=(-i\partial_1\varphi,\ldots,-i\partial_n\varphi),
\\
B\varphi&=(x_1\varphi,\ldots,x_n\varphi)
\end{align}
for any $\varphi\in C_0^\infty(\mathbb{R}^n;\mathbb{C})$, compactly supported smooth functions on $\mathbb{R}^n$.
In fact, the natural domains of $A$ and $B$ are given respectively by
\begin{align}
D(A)
&=H^1(\mathbb{R}^n)
=\{
\varphi\in L^2(\mathbb{R}^n);\,\partial_j\varphi\in L^2(\mathbb{R}^n)\ \text{for all}\ j\ \text{with}\ 1\le j\le n
\},
\\
D(B)
&
=\{
\varphi\in L^2(\mathbb{R}^n);\,x_j\varphi\in L^2(\mathbb{R}^n)\ \text{for all}\ j\ \text{with}\ 1\le j\le n
\},
\end{align}
where the derivatives are understood to be distributional derivatives and $H^1$ denotes the standard Sobolev space of order one.
Since $A\varphi$ and $B\varphi$ are $\mathbb{C}^n$-valued, the corresponding natural Hilbert space is given by $H=L^2(\mathbb{R}^n;\mathbb{C}^n)$ with scalar product
\begin{equation}
((\varphi_1,\ldots,\varphi_n)|(\psi_1,\ldots,\psi_n))
=\sum_{j=1}^n
(\varphi_j|\psi_j)
=\sum_{j=1}^n\int_{\mathbb{R}^n}\varphi_j\overline{\psi_j}\,dx,
\end{equation}
where $\varphi_j,\psi_j\in L^2(\mathbb{R}^n;\mathbb{C})$, $1\le j\le n$.
Since $C_0^\infty(\mathbb{R}^n)$ is dense in $L^2(\mathbb{R}^n)$, all computations will be carried out on $C_0^\infty$ and then on $M\equiv D(A)\cap D(B)$ by density.

\begin{theorem}\label{thm:3.1}
\begin{enumerate}
\item Let $\varphi\in M$ satisfy $x\varphi\neq0$, $\nabla\varphi\neq0$.
Then, we have
\begin{align}
n\|\varphi\|^2
&=-2\Re(x\varphi|\nabla\varphi)
\label{eqn:3.1}
\\
&=\|x\varphi\|\|\nabla\varphi\|
\left(
2-\left\|
\frac{x\varphi}{\|x\varphi\|}
+
\frac{\nabla\varphi}{\|\nabla\varphi\|}
\right\|^2
\right)
\label{eqn:3.2}
\\
&=\|x\varphi\|^2
+\|\nabla\varphi\|^2
-\|x\varphi+\nabla\varphi\|^2,
\label{eqn:3.3}
\end{align}
and 
\begin{equation}
|(x\varphi|\nabla\varphi)|
=\|x\varphi\|\|\nabla\varphi\|
\left(
1-\frac{1}{2}\left\|
\frac{x\varphi}{\|x\varphi\|}
-[\sgn(x\varphi|\nabla\varphi)]\frac{\nabla\varphi}{\|\nabla\varphi\|}
\right\|^2
\right).
\label{eqn:SchroedingerEq}
\end{equation}

\item Let $\varphi\in M$.
Then, the following statements are equivalent:
\medskip
\begin{enumerate}
\item $n\|\varphi\|^2=\|x\varphi\|^2+\|\nabla\varphi\|^2$.

\medskip
\item $x\varphi=-\nabla\varphi$.

\medskip
\item There exists $\theta\in\mathbb{R}$ such that, for any $x\in\mathbb{R}^n$, $\varphi$ satisfies
\begin{equation}
\varphi(x)=e^{i\theta}\frac{1}{\pi^{n/4}}\|\varphi\|\exp\!\left(-\frac{|x|^2}{2}\right).
\label{eqn:Coherent}
\end{equation}
\end{enumerate}

\item Let $\varphi\in M$ satisfy $x\varphi\neq0$.
Then, the following statements are equivalent:
\medskip
\begin{enumerate}
\item $n\|\varphi\|^2=2\|x\varphi\|\|\nabla\varphi\|$.

\medskip
\item $\|\nabla\varphi\|x\varphi=-\|x\varphi\|\nabla\varphi$.

\medskip
\item There exists $\theta\in\mathbb{R}$ such that, for any $x\in\mathbb{R}^n$, $\varphi$ satisfies
\begin{equation}
\varphi(x)=e^{i\theta}
\left(
\frac{\|\nabla\varphi\|}{\pi\|x\varphi\|}
\right)^{n/4}
\|\varphi\|\exp\!\left(-\frac{\|\nabla\varphi\|}{\|x\varphi\|}\frac{|x|^2}{2}\right).
\label{eqn:Squeezed}
\end{equation}
\end{enumerate}

\item Let $\varphi\in M$ satisfy $x\varphi\neq0$.
Then, the following statements are equivalent:
\medskip
\begin{enumerate}
\item $|(x\varphi|\nabla\varphi)|=\|x\varphi\|\|\nabla\varphi\|$.

\medskip
\item $\|\nabla\varphi\|x\varphi=[\sgn(x\varphi|\nabla\varphi)]\|x\varphi\|\nabla\varphi$.

\medskip
\item There exists $\theta\in\mathbb{R}$ such that, for any $x\in\mathbb{R}^n$, $\varphi$ satisfies
\begin{equation}
\varphi(x)=e^{i\theta}
\left(
-[\Re\sgn(x\varphi|\nabla\varphi)]
\frac{\|\nabla\varphi\|}{\pi\|x\varphi\|}
\right)^{n/4}
\|\varphi\|\exp\!\left(
[\sgn\overline{(x\varphi|\nabla\varphi)}]
\frac{\|\nabla\varphi\|}{\|x\varphi\|}\frac{|x|^2}{2}
\right).
\label{eqn:SqueezedGen}
\end{equation}
Note that for $\nabla\varphi\neq0$
\begin{equation}
\Re(x\varphi|\nabla\varphi)
=\Re\overline{(x\varphi|\nabla\varphi)}
=-\frac{n}{2}\|\varphi\|^2<0
\end{equation}
and hence $\Re\sgn(x\varphi|\nabla\varphi)=\Re\sgn\overline{(x\varphi|\nabla\varphi)}<0$.
\end{enumerate}
\end{enumerate}
\end{theorem}

\begin{remark}
In general, one of the statements in Part (2) implies any of the statements in Part (3).
The converse implication holds if and only if $\|x\varphi\|=\|\nabla\varphi\|$.
Moreover, one of the statements in Part (3) implies any of the statements in Part (4).
The converse implication holds if and only if $(x\varphi|\nabla\varphi)=-|(x\varphi|\nabla\varphi)|$.
\end{remark}

\begin{remark}
For $n=1$, the inequality 
\begin{equation}
\|x\varphi\|\|\nabla\varphi\|\ge\frac{1}{2}\|\varphi\|^2
\label{eqn:KennardIneq}
\end{equation}
implied by equality (\ref{eqn:3.2}) is known as the Kennard uncertainty inequality, which is a version of the Robertson uncertainty inequality (\ref{eqn:RobertsonIneq}) specialized to $A=-i\nabla$ and $B=x$.
The vector $\varphi$ saturating the Kennard uncertainty inequality (\ref{eqn:KennardIneq}) is given by (\ref{eqn:Squeezed}), and is known to be a squeezed state in the field of quantum optics \cite{ref:QuantumOptics-Scully,ref:QuantumOptics-WallsMilburn}.
If furthermore the vector $\varphi$ satisfies $\|x\varphi\|=\|\nabla\varphi\|=\frac{1}{\sqrt{2}}\|\varphi\|$, it is reduced to (\ref{eqn:Coherent}), and is called coherent state \cite{ref:QuantumOptics-Scully,ref:QuantumOptics-WallsMilburn}.
A tighter inequality than (\ref{eqn:KennardIneq}),
\begin{equation}
\|x\varphi\|\|\nabla\varphi\|
\ge|(x\varphi|\nabla\varphi)|
=\frac{1}{2}\,\Bigl(
\|\varphi\|^2
+|((x\cdot\nabla+\nabla\cdot x)\varphi|\varphi)|^2
\Bigr)^{1/2},
\label{eqn:SchroedingerIneqXP}
\end{equation}
is available from equality (\ref{eqn:SchroedingerEq}) for $n=1$, as a special version of the Schr\"odinger-Robertson uncertainty inequality (\ref{eqn:SchroedingerIneq}).
Inequality (\ref{eqn:SchroedingerIneqXP}) is saturated by the vector $\varphi$ given in (\ref{eqn:SqueezedGen}), which is again a squeezed state.
The family of the extremizers (\ref{eqn:SqueezedGen}) of the Schr\"odinger-Robertson uncertainty inequality (\ref{eqn:SchroedingerIneqXP}) includes the squeezed state (\ref{eqn:Squeezed}) and the coherent state (\ref{eqn:Coherent}) as special cases.
\end{remark}

\begin{proof}[Proof of Theorem \ref{thm:3.1}]
\begin{enumerate}
\item Let $\varphi\in C_0^\infty$.
Then, 
\begin{equation}
2\Re(x\varphi|\nabla\varphi)
=\int x\cdot\nabla|\varphi|^2\,dx
=-\int(\div x)|\varphi|^2\,dx
=-n\|\varphi\|^2.
\end{equation}
Moreover, we have
\begin{equation}
\|x\varphi+\nabla\varphi\|^2
=\|x\varphi\|^2
+\|\nabla\varphi\|^2
+2\Re(x\varphi|\nabla\varphi).
\end{equation}
Then, equalities (\ref{eqn:3.1})--(\ref{eqn:3.3}) follow, by recalling (\ref{eqn:2.5}) and (2.10)$_+$.
On the other hand, the identity (\ref{eqn:SchroedingerEq}) follows from (\ref{eqn:2.12}), by noting $|{\sgn z}|=1$ and $(\sgn z)^{-1}=\overline{\sgn z}=\sgn\overline{z}$.

\medskip
\item The equivalence between (i) and (ii) follows from (\ref{eqn:3.3}).
If $\varphi$ has the form (iii), then (ii) follows by a direct calculation.
Conversely, if (ii) holds, then
\begin{equation}
\nabla\!\left[
\exp\!\left(
\frac{|x|^2}{2}
\right)
\varphi
\right]
=\exp\!\left(
\frac{|x|^2}{2}
\right)
(x\varphi+\nabla\varphi)=0,
\end{equation}
and therefore, for some $c\in\mathbb{C}$, $\varphi$ is represented as
\begin{equation}
\varphi(x)=c\exp\!\left(
-\frac{|x|^2}{2}
\right).
\end{equation}
Then, (iii) follows by evaluating $\|\varphi\|$.

\medskip
\item The equivalence between (i) and (ii) follows from (\ref{eqn:3.2}).
If $\varphi$ has the form (iii), then (ii) follows by a direct calculation.
Conversely, if (ii) holds, then
\begin{equation}
\nabla\!\left[
\exp\!\left(
\frac{\|\nabla\varphi\|}{\|x\varphi\|}
\frac{|x|^2}{2}
\right)
\varphi
\right]
=
\exp\!\left(
\frac{\|\nabla\varphi\|}{\|x\varphi\|}
\frac{|x|^2}{2}
\right)
\left(
\frac{\|\nabla\varphi\|}{\|x\varphi\|}x\varphi+\nabla\varphi
\right)=0,
\end{equation}
and therefore, for some $c\in\mathbb{C}$, $\varphi$ is represented as
\begin{equation}
\varphi(x)=c\exp\!\left(
-\frac{\|\nabla\varphi\|}{\|x\varphi\|}
\frac{|x|^2}{2}
\right).
\end{equation}
Then, (iii) follows by evaluating $\|\varphi\|$.

\medskip
\item The equivalence between (i) and (ii) follows from (\ref{eqn:SchroedingerEq}).
If $\varphi$ has the form (iii), then (ii) follows by a direct calculation.
Conversely, if (ii) holds, then
\begin{align}
&\nabla\!\left[
\exp\!\left(
-[\sgn\overline{(x\varphi|\nabla\varphi)}]
\frac{\|\nabla\varphi\|}{\|x\varphi\|}\frac{|x|^2}{2}
\right)
\varphi
\right]
\nonumber\\
&\qquad
=
\exp\!\left(
-[\sgn\overline{(x\varphi|\nabla\varphi)}]
\frac{\|\nabla\varphi\|}{\|x\varphi\|}\frac{|x|^2}{2}
\right)
\left(
-[\sgn\overline{(x\varphi|\nabla\varphi)}]
\frac{\|\nabla\varphi\|}{\|x\varphi\|}x\varphi+\nabla\varphi
\right)
=0
\end{align}
by noting $\sgn\overline{z}=(\sgn z)^{-1}$, and therefore, for some $c\in\mathbb{C}$, $\varphi$ is represented as
\begin{equation}
\varphi(x)=c\exp\!\left(
[\sgn\overline{(x\varphi|\nabla\varphi)}]
\frac{\|\nabla\varphi\|}{\|x\varphi\|}
\frac{|x|^2}{2}
\right).
\end{equation}
Then, (iii) follows by evaluating $\|\varphi\|$.
\end{enumerate}
\end{proof}

We now rewrite (\ref{eqn:3.1}) as
\begin{equation}
\Re\Bigl(x\cdot\nabla\varphi+\frac{n}{2}\varphi\Big|\varphi\Bigr)=0
\label{eqn:3.4}
\end{equation}
and regard (\ref{eqn:3.4}) as an orthogonality relation.
Then, as in \cite{ref:MachiharaOzawaWadadeSubmitted} we notice that (\ref{eqn:3.4}) yields a new equality:
\begin{theorem}\label{thm:3.2}
The following equality
\begin{equation}
\|x\cdot\nabla\varphi\|^2
=\Bigl\|x\cdot\nabla\varphi+\frac{n}{2}\varphi\Bigr\|^2
+\left(\frac{n}{2}\right)^2\|\varphi\|^2
\label{eqn:3.5}
\end{equation}
holds for all $\varphi\in D(A)$ with $x\cdot\nabla\varphi\in L^2(\mathbb{R}^n)$.
There does not exist $\varphi\in D(A)\setminus\{0\}$ with $x\cdot\nabla\varphi\in L^2(\mathbb{R}^n)$ satisfying
\begin{equation}
\|x\cdot\nabla\varphi\|^2
=\left(\frac{n}{2}\right)^2\|\varphi\|^2.
\label{eqn:3.6}
\end{equation}
\end{theorem}

\begin{remark}\label{rmk:3.2}
The inequality
\begin{equation}
\frac{n}{2}\|\varphi\|\le\|x\cdot\nabla\varphi\|
\label{eqn:3.7}
\end{equation}
as well as the nonexistence of nontrivial extremizers has been proved in \cite{ref:OzawaSasaki}.
Theorem \ref{thm:3.2} is recognized as an optimal description of (\ref{eqn:3.7}) from the point of view of equalities.
\end{remark}

\begin{proof}[Proof of Theorem \ref{thm:3.2}]
The equality (\ref{eqn:3.5}) follows from (\ref{eqn:3.4}), by noticing that
\begin{align}
\|x\cdot\nabla\varphi\|^2
&=\Bigl\|
\Bigl(
x\cdot\nabla\varphi+\frac{n}{2}\varphi
\Bigr)
-\frac{n}{2}\varphi
\Bigr\|^2
\nonumber\\
&=\Bigl\|
x\cdot\nabla\varphi+\frac{n}{2}\varphi
\Bigr\|^2
+\left(\frac{n}{2}\right)^2
\|\varphi\|^2.
\end{align}
Moreover, since
\begin{equation}
x\cdot\nabla\varphi+\frac{n}{2}\varphi
=|x|^{-n/2}x\cdot\nabla(|x|^{n/2}\varphi),
\end{equation}
(\ref{eqn:3.6}) holds if and only if there exists a function $f:S^{n-1}\to\mathbb{C}$ satisfying
\begin{equation}
\varphi(x)=|x|^{-n/2}f\Bigl(\frac{x}{|x|}\Bigr)
\end{equation}
for all $x\in\mathbb{R}^n\setminus\{0\}$, where $S^{n-1}$ is the unit sphere $S^{n-1}=\{x\in\mathbb{R}^n;|x|=1\}$.
Then, $\varphi\in L^2(\mathbb{R}^n)$ if and only if
\begin{equation}
\int_{S^{n-1}}|f(\omega)|^2\,d\sigma(\omega)=0,
\end{equation}
which in turn is equivalent to $f=0$ and to $\varphi=0$, where $\sigma$ is the surface element, namely, the Lebesgue measure on $S^{n-1}$.
This proves the nonexistence of nontrivial extremizers of (\ref{eqn:3.7}).
\end{proof}

In \cite{ref:OzawaSasaki}, it has been proved that the inequality (\ref{eqn:3.7}) is equivalent to the standard Hardy inequality for $n\ge3$.
The following theorem describes such relationship at the level of equalities:
\begin{theorem}\label{thm:3.3}
Let $n\ge3$.
Then, the equality
\begin{equation}
\Bigl\|\frac{x}{|x|}\cdot\nabla\psi\Bigr\|^2
=\Bigl\|
\frac{x}{|x|}\cdot\nabla\psi
+\frac{n-2}{2|x|}\psi
\Bigr\|^2
+\left(
\frac{n-2}{2}
\right)^2
\Bigl\|\frac{\psi}{|x|}\Bigr\|^2
\label{eqn:3.8}
\end{equation}
for all $\psi\in D(A)$ follows from (\ref{eqn:3.5}).
Conversely, (\ref{eqn:3.8}) implies (\ref{eqn:3.5}).
\end{theorem}

\begin{proof}
Since $C_0^\infty(\mathbb{R}^n\setminus\{0\})$ is dense in $H^1(\mathbb{R}^n)=D(A)$ for $n\ge3$, we may assume that $\varphi,\psi\in C_0^\infty(\mathbb{R}^n\setminus\{0\})$. 
The following calculations are justified as long as $\varphi,\psi\in C_0^\infty(\mathbb{R}^n\setminus\{0\})$ without restriction on the space dimensions.
First, suppose that (\ref{eqn:3.5}) holds for all $\varphi\in C_0^\infty(\mathbb{R}^n\setminus\{0\})$.
Let $\psi\in C_0^\infty(\mathbb{R}^n\setminus\{0\})$.
Then, we define $\varphi$ by $\varphi(x)=\frac{1}{|x|}\psi(x)$, $x\in\mathbb{R}^n\setminus\{0\}$.
It follows that $\varphi\in C_0^\infty(\mathbb{R}^n\setminus\{0\})$ and the left-hand side of (\ref{eqn:3.5}) is rewritten as
\begin{align}
\|x\cdot\nabla\varphi\|^2
&=\Bigl\|
x\cdot\nabla\frac{\psi}{|x|}
\Bigr\|^2
\nonumber\\
&=\Bigl\|
\frac{x}{|x|}\cdot\nabla\psi
-\frac{1}{|x|}\psi
\Bigr\|^2
\nonumber\\
&=
\Bigl\|
\frac{x}{|x|}\cdot\nabla\psi
\Bigr\|^2
-2\Re\int\frac{x}{|x|}\cdot\nabla\psi
\,
\overline{\frac{1}{|x|}\psi}
\,dx
+\Bigl\|
\frac{\psi}{|x|}
\Bigr\|^2
\nonumber\\
&=
\Bigl\|
\frac{x}{|x|}\cdot\nabla\psi
\Bigr\|^2
-\int\frac{x}{|x|^2}\cdot\nabla|\psi|^2\,dx
+\Bigl\|
\frac{\psi}{|x|}
\Bigr\|^2
\nonumber\\
&=
\Bigl\|
\frac{x}{|x|}\cdot\nabla\psi
\Bigr\|^2
+\int\left(\div\frac{x}{|x|^2}\right)
|\psi|^2\,dx
+\Bigl\|
\frac{\psi}{|x|}
\Bigr\|^2
\nonumber\\
&=
\Bigl\|
\frac{x}{|x|}\cdot\nabla\psi
\Bigr\|^2
+(n-1)
\Bigl\|
\frac{\psi}{|x|}
\Bigr\|^2,
\label{eqn:3.9}
\end{align}
where we have used Gauss' divergence theorem, while the first term on the right-hand side of (\ref{eqn:3.5}) is rewritten as
\begin{align}
\Bigl\|
x\cdot\nabla\varphi
+\frac{n}{2}\varphi
\Bigr\|^2
&=
\Bigl\|
x\cdot\nabla\frac{\psi}{|x|}
+\frac{n}{2}\frac{\psi}{|x|}
\Bigr\|^2
\nonumber\\
&=
\Bigl\|
\frac{x}{|x|}\cdot\nabla\psi
+\frac{n-2}{2|x|}\psi
\Bigr\|^2.
\label{eqn:3.10}
\end{align}
Combining (\ref{eqn:3.9}) and (\ref{eqn:3.10}), we derive (\ref{eqn:3.8}) from (\ref{eqn:3.5}) with $\varphi=\frac{1}{|x|}\psi$, noticing that
\begin{equation}
\left(\frac{n}{2}\right)^2-(n-1)=\left(\frac{n-2}{2}\right)^2.
\end{equation}
Conversely, suppose that (\ref{eqn:3.8}) holds for all $\psi\in C_0^\infty(\mathbb{R}^n\setminus\{0\})$.
Let $\varphi\in C_0^\infty(\mathbb{R}^n\setminus\{0\})$.
Then, we define $\psi=|x|\varphi$, $x\in\mathbb{R}^n\setminus\{0\}$.
It follows that $\psi\in C_0^\infty(\mathbb{R}^n\setminus\{0\})$ and all the computations in (\ref{eqn:3.9}) and (\ref{eqn:3.10}) can be traced backward to imply (\ref{eqn:3.5}).
\end{proof}

In \cite{ref:ReedSimonII}, the standard Hardy type inequalities of the form
\begin{equation}
\Bigl\|
\frac{\psi}{|x|}
\Bigr\|
\le\frac{2}{n-2}\,\Bigl\|
\frac{x}{|x|}\cdot\nabla\psi
\Bigr\|
\le\frac{2}{n-2}\|
\nabla\psi
\|
\label{eqn:3.11}
\end{equation}
are referred to as the uncertainty principle lemma.
Here we have derived (\ref{eqn:3.11}) as a corollary to (\ref{eqn:3.4}), which is equivalent to (\ref{eqn:3.1}), which in turn is regarded as an original form of the uncertainty relation between the position and momentum operators.

\subsection{Generator of Dilations and Free Hamiltonian}
Let $A=\frac{1}{2i}(x\cdot\nabla+\nabla\cdot x)=-ix\cdot\nabla-i\frac{n}{2}$ and $B=-\Delta=-\nabla\cdot\nabla$ with natural domains
\begin{align}
D(A)
&
=\{
\varphi\in H^1(\mathbb{R}^n);\,
x\cdot\nabla\varphi\in L^2(\mathbb{R}^n)\},
\\
D(B)
&
=H^2(\mathbb{R}^n)
=\{
\varphi\in L^2(\mathbb{R}^n);\,
\partial_j\partial_k\varphi\in L^2(\mathbb{R}^n)\ \text{for all}\ j,k\ \text{with}\ 1\le j,k\le n
\}.
\end{align}
The operator $A$ is called the generator of dilations in the sense that one-parameter group of dilations $\{T(\theta); \theta\in\mathbb{R}\}$ defined by
\begin{equation}
(T(\theta)\varphi)(x)
=e^{\frac{n}{2}\theta} \varphi (e^\theta x),\qquad
x\in\mathbb{R}^n
\end{equation}
satisfies 
\begin{equation}
T'(0)\varphi
=\frac{d}{d\theta}T(\theta)\varphi\biggr|_{\theta=0}
=iA\varphi.
\end{equation}
The generator of dilations $A$ and the free Hamiltonian $B$ have a special commutation relation
\begin{equation}
[A,B]\varphi=2iB\varphi
\label{eqn:3.14}
\end{equation}
for smooth functions $\varphi\in C_0^\infty(\mathbb{R}^n;\mathbb{C})$, where the commutator is understood to be $AB-BA$ since $C_0^\infty(\mathbb{R}^n)\subset D(BA)\cap D(AB)\subset D(A)\cap D(B)$.

\begin{theorem}\label{thm:3.4}
Let  $\varphi\in M\equiv D(A)\cap D(B)$ satisfy $A\varphi\neq0$, $B\varphi\neq0$.
Then, 
\begin{align}
2\|\nabla\varphi\|^2
&=2(B\varphi|\varphi)
\label{eqn:3.15a}
\\
&
=-i([A,B]\varphi|\varphi)
\label{eqn:3.15b}
\\
&=-2\Im(A\varphi|B\varphi)
\label{eqn:3.15c}
\\
&=\|A\varphi\|\|B\varphi\|\left(
2-\left\|
\frac{A\varphi}{\|A\varphi\|}
+i\frac{B\varphi}{\|B\varphi\|}
\right\|^2
\right).
\label{eqn:3.15}
\end{align}
\end{theorem}

\begin{remark}
As a direct consequence, we have the inequality
\begin{equation}
\|\nabla\varphi\|^2
\le\Bigl\|
x\cdot\nabla\varphi+\frac{n}{2}\varphi\Bigr\|
\|\Delta\varphi\|,
\label{eqn:3.16}
\end{equation}
which might be new.
The inequality (\ref{eqn:3.16}) relates the information given by the momentum operator, the generator of the dilations, and the free Hamiltonian.
\end{remark}

\begin{remark}
By (\ref{eqn:3.5}), we already know that
\begin{equation}
\|A\varphi\|^2
=
\|x\cdot\nabla\varphi\|^2
-\left(\frac{n}{2}\right)^2\|\varphi\|^2,
\end{equation}
which implies (\ref{eqn:3.7}) directly, as stated in Remark \ref{rmk:3.2}.
\end{remark}

\begin{proof}[Proof of Theorem \ref{thm:3.4}]
The equalities in (\ref{eqn:3.15a})--(\ref{eqn:3.15}) follow from (\ref{eqn:3.14}), (\ref{eqn:2.5}), (2.10)$_-$, and the equality
\begin{equation}
\|\nabla\varphi\|^2
=-(\Delta\varphi|\varphi).
\end{equation}
\end{proof}

\subsection{Radial Derivative and Coulomb Potential}
Let $n\ge3$ as in Theorem \ref{thm:3.3}.
Let $A=\frac{1}{2i}(\frac{x}{|x|}\cdot\nabla+\nabla\cdot\frac{x}{|x|})=-i\frac{x}{|x|}\cdot\nabla-i\frac{n-1}{2|x|}$ and $B=\frac{1}{|x|}$ with natural domains
\begin{align}
D(A)
&
=\{
\varphi\in L^2(\mathbb{R}^n);\,
\tfrac{x}{|x|}\cdot\nabla\varphi,\tfrac{1}{|x|}\varphi\in L^2(\mathbb{R}^n)\},
\\
D(B)
&
=\{
\varphi\in L^2(\mathbb{R}^n);\,
\tfrac{1}{|x|}\varphi\in L^2(\mathbb{R}^n)
\}.
\end{align}
The operator $A$ is regarded as a symmetrized radial derivative defined by $\partial_r\equiv\frac{x}{|x|}\cdot\nabla$.
The squared length of the gradient has a pointwise decomposition
\begin{equation}
|\nabla\varphi|^2
=|\partial_r\varphi|^2
+\sum_{j=1}^n|L_j\varphi|^2,
\label{eqn:3.17}
\end{equation}
where $L_j$ is the $j$th component of the spherical derivative defined by
\begin{equation}
L_j\equiv\partial_j-\frac{x_j}{|x|}\partial_r.
\end{equation}
At the point $x\in\mathbb{R}^n$, the unit outer vector is given by $\frac{x}{|x|}$ and it is orthogonal to $e_j-\frac{x_j}{|x|}\frac{x}{|x|}$, $1\le j\le n$, where $e_j=(0,\ldots,0,1,0,\ldots,0)$ is the standard unit vector in the $j$th direction.
The corresponding one-parameter family of operators acting on functions are given by
\begin{align}
(T(\theta)\varphi)(x)&=\varphi\Bigl(x+\theta\frac{x}{|x|}\Bigr),\\
(T_j(\theta)\varphi)(x)&=\varphi\Bigl(x+\theta\,\Bigl(e_j-\frac{x_j}{|x|}\frac{x}{|x|}\Bigr)\Bigr),
\end{align}
which satisfy
\begin{align}
T'(0)\varphi=\frac{d}{d\theta}T(\theta)\varphi\biggr|_{\theta=0}
=\partial_r\varphi,\\
T_j'(0)\varphi=\frac{d}{d\theta}T_j(\theta)\varphi\biggr|_{\theta=0}
=L_j\varphi.
\end{align}
The (symmetrized) radial derivative and the Coulomb potential have a special commutation relation
\begin{equation}
[A,B]\varphi=iB^2\varphi
\label{eqn:3.18}
\end{equation}
for smooth functions $\varphi\in C_0^\infty(\mathbb{R}^n\setminus\{0\};\mathbb{C})$, where the commutator is understood to be $AB-BA$ since $C_0^\infty(\mathbb{R}^n\setminus\{0\})\subset D(BA)\cap D(AB)\subset D(A)\cap D(B)$.

\begin{theorem}\label{thm:3.5}
Let $n\ge3$ and let $\varphi\in H^1(\mathbb{R}^n)$ satisfy $A\varphi\neq0$, $B\varphi\neq0$.
Then, 
\begin{align}
\|B\varphi\|^2
&=-i([A,B]\varphi|\varphi)
\label{eqn:3.19a}
\\
&=-2\Im(A\varphi|B\varphi)
\label{eqn:3.19b}
\\
&=\|A\varphi\|\|B\varphi\|\left(
2-\left\|
\frac{A\varphi}{\|A\varphi\|}
+i\frac{B\varphi}{\|B\varphi\|}
\right\|^2
\right),
\label{eqn:3.19}
\\
\|A\varphi\|^2
&=\|\partial_r\varphi\|^2
-\frac{(n-1)(n-3)}{4}\|B\varphi\|^2.
\label{eqn:3.20}
\end{align}
\end{theorem}

\begin{remark}
As a direct consequence of (\ref{eqn:3.19}), we have the inequality
\begin{equation}
\Bigl\|\frac{1}{|x|}\varphi\Bigr\|
=\|B\varphi\|
\le2\|A\varphi\|
=2\,\Bigl\|
\partial_r\varphi+\frac{n-1}{2|x|}\varphi\Bigr\|,
\label{eqn:3.21}
\end{equation}
which might be new.
The inequality (\ref{eqn:3.21}) relates the information given by the radial derivative and Coulomb potential.
To be more specific, (\ref{eqn:3.21}) shows that the Coulomb potential $B$ is $A$ (symmetrized radial derivative)-bounded with relative bound 2.
\end{remark}

\begin{proof}[Proof of Theorem \ref{thm:3.5}]
The equalities in (\ref{eqn:3.19a})--(\ref{eqn:3.19}) follow from (\ref{eqn:3.18}), (\ref{eqn:2.5}), and (2.10)$_-$.
In the same way as in the proof of Theorem \ref{thm:3.3} we calculate 
\begin{align}
\|A\varphi\|^2
&=\Bigl\|
\frac{x}{|x|}\cdot\nabla\varphi+\frac{n-1}{2|x|}\varphi\Bigr\|^2
\nonumber\\
&=\Bigl\|\frac{x}{|x|}\cdot\nabla\varphi\Bigr\|^2
+(n-1)\Re\int\frac{x}{|x|}\cdot\nabla\varphi\,\overline{\frac{1}{|x|}\varphi}\,dx
+\left(
\frac{n-1}{2}
\right)^2
\,\Bigl\|\frac{1}{|x|}\varphi\Bigr\|^2
\nonumber\\
&=\Bigl\|\frac{x}{|x|}\cdot\nabla\varphi\Bigr\|^2
+\frac{n-1}{2}\int\frac{x}{|x|^2}\cdot\nabla|\varphi|^2\,dx
+\left(
\frac{n-1}{2}
\right)^2
\,\Bigl\|\frac{1}{|x|}\varphi\Bigr\|^2
\nonumber\\
&=\Bigl\|\frac{x}{|x|}\cdot\nabla\varphi\Bigr\|^2
-\frac{n-1}{2}\int\left(\div\frac{x}{|x|^2}\right)|\varphi|^2\,dx
+\left(
\frac{n-1}{2}
\right)^2
\,\Bigl\|\frac{1}{|x|}\varphi\Bigr\|^2
\end{align}
to obtain (\ref{eqn:3.20}).
\end{proof}

We now rewrite (\ref{eqn:3.18}) or (\ref{eqn:3.19b}) as
\begin{equation}
\Re(B\varphi-2i A\varphi|B\varphi)=0
\label{eqn:3.22}
\end{equation}
and regard (\ref{eqn:3.22}) as an orthogonality relation.
Then, as in \cite{ref:MachiharaOzawaWadadeSubmitted} we notice that (\ref{eqn:3.22}) yields a new equality,
\begin{align}
4\|A\varphi\|^2
&=\|
(2iA\varphi-B\varphi)+B\varphi
\|^2
\nonumber\\
&=\|2iA\varphi-B\varphi||^2
+\|B\varphi\|^2,
\end{align}
where the right-hand side is exactly the same as
\begin{equation}
4\,\Bigl\|
\frac{x}{|x|}\cdot\nabla\varphi
+\frac{n-2}{2|x|}\varphi
\Bigr\|^2
+\Bigl\|\frac{1}{|x|}\varphi\Bigr\|^2,
\end{equation}
while the left-hand side is rewritten as the right-hand side of (\ref{eqn:3.20}).
Therefore, we have proved 
\begin{equation}
4\|\partial_r\varphi\|^2
=4\,\Bigl\|
\frac{x}{|x|}\cdot\nabla\varphi
+\frac{n-2}{2|x|}\varphi
\Bigr\|^2
+(n-2)^2\,\Bigl\|
\frac{1}{|x|}\varphi
\Bigr\|^2,
\label{eqn:3.24}
\end{equation}
which is exactly the same as (\ref{eqn:3.8}).
We have thus derived the standard Hardy type inequality (the uncertainty principle lemma \cite{ref:ReedSimonII}) from the orthogonality (\ref{eqn:3.22}), which is equivalent to the commutation relation between the radial derivative and Coulomb potential (\ref{eqn:3.18}).
By (\ref{eqn:3.17}), the equality (\ref{eqn:3.8}) or (\ref{eqn:3.24}) is also rewritten as
\begin{align}
\|\nabla\varphi\|^2
-\sum_{j=1}^n\|L_j\varphi\|^2
&=\|\partial_r\varphi\|^2
\nonumber\\
&=\left(\frac{n-2}{2}\right)^2
\Bigl\|\frac{1}{|x|}\varphi\Bigr\|^2
+\Bigl\|\frac{x}{|x|}\cdot\nabla\varphi+\frac{n-2}{2|x|}\varphi
\Bigr\|^2.
\end{align}

\appendix
\renewcommand{\thetheorem}{\Alph{section}.\arabic{theorem}}
\renewcommand{\theremark}{\Alph{section}.\arabic{remark}}
\renewcommand{\thecorollary}{\Alph{section}.\arabic{corollary}}
\section{Basics of the Cauchy-Schwarz Inequality}
In this appendix, we summarize algebraic identities related to the Cauchy-Schwarz inequality.
For this purpose, we introduce sign function $\sgn:\mathbb{C}\to\mathbb{R}$ by
\begin{equation}
\sgn z
=\begin{cases}
\medskip
\displaystyle
z/|z|,&z\in\mathbb{C}\setminus\{0\},\\
\displaystyle
1,&z=0.
\end{cases}
\end{equation}

\begin{theorem}
The following equalities hold for all $u,v\in H\setminus\{0\}$:
\begin{align}
|(u|v)|
&=\|u\|\|v\|\left(
1-\frac{1}{2}\left\|
\frac{u}{\|u\|}
-[\sgn(u|v)]\frac{v}{\|v\|}
\right\|^2
\right),
\label{eqn:A.1}
\\
\pm\Re(u|v)
&=\|u\|\|v\|\left(
1-\frac{1}{2}\left\|
\frac{u}{\|u\|}
\mp\frac{v}{\|v\|}
\right\|^2
\right),
\stepcounter{equation}\tag*{(\theequation)$_\pm$}
\label{eqn:A.2}
\\
\pm\Im(u|v)
&=\|u\|\|v\|\left(
1-\frac{1}{2}\left\|
\frac{u}{\|u\|}
\mp i\frac{v}{\|v\|}
\right\|^2
\right).
\stepcounter{equation}\tag*{(\theequation)$_\pm$}
\label{eqn:A.3}
\end{align}
\end{theorem}
\begin{proof}
Let $(u|v)\neq0$.
Then, we expand the square of the last norm on the right-hand side of (\ref{eqn:A.1}) as 
\begin{align}
\left\|
\frac{u}{\|u\|}
-\frac{(u|v)}{|(u|v)|}\frac{v}{\|v\|}
\right\|^2
&=
\left\|
\frac{u}{\|u\|}
\right\|^2
-2\Re\!\left(\frac{u}{\|u\|}\bigg|\frac{(u|v)}{|(u|v)|}\frac{v}{\|v\|}
\right)
+\left\|
\frac{(u|v)}{|(u|v)|}\frac{v}{\|v\|}
\right\|^2
\nonumber\\
&=
2-2\frac{\Re[\overline{(u|v)}(u|v)]}{\|u\|\|v\||(u|v)|}
\nonumber\\
&=
2-2\frac{|(u|v)|}{\|u\|\|v\|},
\label{eqn:A.5}
\end{align}
which yields (\ref{eqn:A.1}).
If $(u|v)=0$, then a similar calculation yields (\ref{eqn:A.1}) in the trivial case.
Equalities \ref{eqn:A.2} can be proved in the same manner as (\ref{eqn:A.5}).
Or they follow from (\ref{eqn:A.1}) by regarding $H$ as a real vector space with scalar product
\begin{equation}
\Re(\,\cdot\,|\,\cdot\,)
:
H\times H\ni(u,v)
\mapsto
\Re(u|v)\in\mathbb{R},
\end{equation}
since the new scalar product $\Re(u|v)$ satisfies $\sgn\Re(u|v)=1$ if and only if $\Re(u|v)\ge0$ while $\sgn\Re(u|v)=-1$ if and only if $-{\Re(u|v)}=|{\Re(u|v)}|>0$. 
Substituting $v$ by $iv$ in \ref{eqn:A.2} implies \ref{eqn:A.3}.
\end{proof}

\begin{remark}
Equality (\ref{eqn:A.1}) is regarded as an equality version of the Cauchy-Schwarz inequality
\begin{equation}
|(u|v)|\le\|u\|\|v\|.
\end{equation}
Indeed, the former immediately implies the latter.
\end{remark}

\begin{remark}
Equalities (\ref{eqn:A.1}), (A.3)$_+$, and (A.4)$_+$ have been noticed by Aldaz \cite{ref:Aldaz2008,ref:Aldaz2009} and are verified by similar and simpler calculations as above, too.
See also \cite{ref:DuistermaatKolk2004,ref:FujiwaraOzawa} for related subjects.
\end{remark}

\begin{corollary}\label{cor:A.1}
The following equalities hold for all $u,v\in H\setminus\{0\}$:
\begin{align}
|(u|v)|
&=\|u\|\|v\|\left[
\left(
1-\frac{1}{2}\left\|
\frac{u}{\|u\|}
\pm\frac{v}{\|v\|}
\right\|^2
\right)^2
+\left(
1-\frac{1}{2}\left\|
\frac{u}{\|u\|}
\pm i\frac{v}{\|v\|}
\right\|^2
\right)^2
\right]^{1/2}
\stepcounter{equation}\tag*{(\theequation)$_\pm$}
\label{eqn:A.4}
\\
&=\|u\|\|v\|\left[
\left(
1-\frac{1}{2}\left\|
\frac{u}{\|u\|}
\pm\frac{v}{\|v\|}
\right\|^2
\right)^2
+\left(
1-\frac{1}{2}\left\|
\frac{u}{\|u\|}
\mp i\frac{v}{\|v\|}
\right\|^2
\right)^2
\right]^{1/2}
\stepcounter{equation}\tag*{(\theequation)$_\pm$}.
\end{align}
\end{corollary}
\begin{proof}
The corollary follows from \ref{eqn:A.2}, \ref{eqn:A.3}, and the equality
\begin{equation}
|(u|v)|
=\Bigl([\Re(u|v)]^2+[\Im(u|v)]^2\Bigr)^{1/2}.
\end{equation}
\end{proof}

Corollary \ref{cor:A.1} is further generalized as:
\begin{corollary}\label{cor:A.2}
The following equality holds for all $u,v\in H\setminus\{0\}$ and $\theta\in\mathbb{R}$:
\begin{equation}
|(u|v)|
=\|u\|\|v\|\left[
\left(
1-\frac{1}{2}\left\|
\frac{u}{\|u\|}
+e^{i\theta}\frac{v}{\|v\|}
\right\|^2
\right)^2
+\left(
1-\frac{1}{2}\left\|
\frac{u}{\|u\|}
\pm ie^{i\theta}\frac{v}{\|v\|}
\right\|^2
\right)^2
\right]^{1/2}.
\stepcounter{equation}\tag*{(\theequation)$_\pm$}
\label{eqn:A.6}
\end{equation}
\end{corollary}
\begin{proof}
Substituting $v$ by $e^{i\theta}v$ in (A.8)$_+$ [resp.\ (A.9)$_+$] yields (A.11)$_+$ [resp.\ (A.11)$_-$].
\end{proof}

\begin{remark}
Equality (A.8)$_+$ [resp.\ (A.9)$_+$] follows from (A.11)$_+$ [resp.\ (A.11)$_-$] with $\theta\in2\pi\mathbb{Z}$, while equality (A.8)$_-$ [resp.\ (A.9)$_-$] follows from (A.11)$_+$ [resp.\ (A.11)$_-$] with $\theta\in(2\mathbb{Z}+1)\pi$.
\end{remark}

\begin{theorem}
Let $u,v\in H$.
Then, the statements in each of the following Parts (1)--(5) are equivalent:
\begin{enumerate}
\item
\begin{enumerate}
\item $\Re(u|v)=\pm\|u\|\|v\|$.
\medskip
\item $\|v\|u=\pm\|u\|v$.
\medskip
\item $(u|v)=\pm\|u\|\|v\|$.
\end{enumerate}

\medskip
\item
\begin{enumerate}
\item $\Im(u|v)=\pm\|u\|\|v\|$.
\medskip
\item $\|v\|u=\pm i\|u\|v$.
\medskip
\item $(u|v)=\pm i\|u\|\|v\|$.
\end{enumerate}

\medskip
\item
\begin{enumerate}
\item $|{\Re(u|v)}|=\|u\|\|v\|$.
\medskip
\item $\Im(u|v)=0$, $|(u|v)|=\|u\|\|v\|$.
\medskip
\item $\|v\|^2u=[\Re(u|v)]v$.
\medskip
\item $\|u\|^2v=[\Re(u|v)]u$.
\end{enumerate}

\medskip
\item
\begin{enumerate}
\item $|{\Im(u|v)}|=\|u\|\|v\|$.
\medskip
\item $\Re(u|v)=0$, $|(u|v)|=\|u\|\|v\|$.
\medskip
\item $\|v\|^2u=[i\Im(u|v)]v$.
\medskip
\item $\|u\|^2v=-[i\Im(u|v)]u$.
\end{enumerate}

\medskip
\item
\begin{enumerate}
\item $|(u|v)|=\|u\|\|v\|$.
\medskip
\item $\|v\|u=[\sgn(u|v)]\|u\|v$.
\medskip
\item $\|v\|^2u=(u|v)v$.
\medskip
\item $\|u\|^2v=\overline{(u|v)}u$.
\end{enumerate}
\end{enumerate}
\end{theorem}
\begin{proof}
If $u=0$ or $v=0$, then all of the equalities in the theorem trivially hold.
Therefore, we assume that $u\neq0$ and $v\neq0$.
\begin{enumerate}
\item The equivalence between (i) and (ii) follows from \ref{eqn:A.2}.
Given (ii), we calculate
\begin{equation}
\|v\|(u|v)
=(\|v\|u|v)
=(\pm\|u\|v|v)
=\pm\|u\|\|v\|^2,
\end{equation}
which implies (iii) by dividing both sides by $\|v\|>0$.
Finally, (iii) implies (i) by taking the real part of $(u|v)$.

\medskip
\item Part (2) follows from Part (1) by replacing $v$ by $iv$.

\medskip
\item[(5)] The equivalence between (i) and (ii) follows from (\ref{eqn:A.1}).
Given (i) and (ii), a direct calculation yields (iii).
Given (iii), we calculate
\begin{equation}
\|v\|^2\|u\|^2
=(\|v\|^2u|u)
=((u|v)v|u)
=(u|v)(v|u)
=|(u|v)|^2,
\end{equation}
which implies (i) by taking its square root.
This proves the equivalence among (i)--(iii).
A similar argument shows the equivalence among (i), (ii) and (iv), or it follows by exchanging $u$ and $v$ in the preceding argument.

\medskip
\item Given (i), we have the Cauchy-Schwarz inequality
\begin{equation}
\|u\|\|v\|
=|{\Re(u|v)}|
\le\Bigl([\Re(u|v)]^2+[\Im(u|v)]^2\Bigr)^{1/2}
=|(u|v)|
\le\|u\|\|v\|,
\end{equation}
where all those inequalities turn out to be equalities, which in turn imply (ii).
Given (ii), we have by Part (5), which proved above
\begin{equation}
\|v\|^2u=(u|v)v,
\end{equation}
where the imaginary part of $(u|v)$ vanishes to imply (iii).
Given (iii), we calculate
\begin{equation}
\|v\|^2\|u\|^2
=(\|v\|^2u|u)
=([\Re(u|v)]v|u)
=[\Re(u|v)](v|u)
\end{equation}
and take its real part to obtain (i).
This proves the equivalence among (i)--(iii).
A similar argument shows the equivalence among (i), (ii), and (iv), or it follows by exchanging $u$ and $v$ in the preceding argument.

\medskip
\item Part (4) follows from Part (3) by replacing $v$ by $iv$.
\end{enumerate}
\end{proof}

\section*{Acknowledgments}
This work is partially supported by the Top Global University Project from the Ministry of Education, Culture, Sports, Science and Technology (MEXT), Japan.
TO is supported by a Grant-in-Aid for Scientific Research (A) (No.\ 26247014) from Japan Society for the Promotion of Science (JSPS).
KY is supported by a Grant-in-Aid for Scientific Research (C) (No.\ 26400406) from JSPS and by Waseda University Grants for Special Research Projects (No.\ 2015K-202 and No.\ 2016K-215).


\end{document}